\documentclass[11pt]{article}
\usepackage{amsmath, amssymb, amsthm, verbatim,enumerate,bbm}
\usepackage{indentfirst}

\title{Closing gaps in problems related to Hamilton cycles in random graphs and hypergraphs}

\author{ Asaf Ferber
\thanks{Department of Mathematics, Yale University and Department of Mathematics, MIT. Emails:
asaf.ferber@yale.edu and ferbera@mit.edu.}}

\date{\today}
\allowdisplaybreaks

\theoremstyle{plain}
\newtheorem{theorem}{Theorem}[section]
\newtheorem{lemma}[theorem]{Lemma}

\newtheorem{corollary}[theorem]{Corollary}

\begin{document}
\maketitle

\begin{abstract}

We show how to adjust a very nice coupling argument due to McDiarmid
in order to prove/reprove in a novel way results concerning Hamilton
cycles in various models of random graph and hypergraphs. In
particular, we firstly show that for $k\geq 3$, if $pn^{k-1}/\log n$
tends to infinity, then a random $k$-uniform hypergraph on $n$
vertices, with edge probability $p$, with high probability (w.h.p.)
contains a loose Hamilton cycle, provided that $(k-1)|n$. This
generalizes results of Frieze, Dudek and Frieze, and reproves a
result of Dudek, Frieze, Loh and Speiss. Secondly, we show that
there exists $K>0$ such for every $p\geq (K\log n)/n$ the following
holds: Let $G_{n,p}$ be a random graph on $n$ vertices with edge
probability $p$, and suppose that its edges are being colored with
$n$ colors uniformly at random. Then, w.h.p\ the resulting graph
contains a Hamilton cycle with for which all the colors appear (a
rainbow Hamilton cycle). Bal and Frieze proved the latter statement
for graphs on an even number of vertices, where for odd $n$ their
$p$ was $\omega((\log n)/n)$. Lastly, we show that for
$p=(1+o(1))(\log n)/n$, if we randomly color the edge set of a
random directed graph $D_{n,p}$ with $(1+o(1))n$ colors, then
w.h.p.\ one can find a rainbow Hamilton cycle where all the edges
are directed in the same way.

\end{abstract}

\section{Introduction}

In this paper we show how to adjust a very nice coupling argument
due to McDiarmid \cite{McD} in order to prove/reprove problems
related to the existence of Hamilton cycles in various random
grpahs/hypergraphs models. The first problem we consider is related
to the existence of a loose Hamilton cycle in a random $k$-uniform
Hypergraph.

A $k$-uniform hypergraph is a pair $\mathcal H=(V,\mathcal E)$,
where $V$ is the set of \emph{vertices} and $\mathcal E\subseteq
\binom{n}{k}$ is the set of \emph{edges}. In the special case where
$k=2$ we simply refer to it as a \emph{graph} and denote it by
$G=(V,E)$. The random $k$-uniform hypergraph $H^{(k)}_{n,p}$ is
defined by adding each possible edge with probability $p$
independently at random, where for the case $k=2$ we denote it by
$G_{n,p}$ (the usual binomial random graph). We define a
\emph{loose} Hamilton cycle as a cyclic ordering of $V$ for which
the edges consist of $k$ consecutive vertices, and for each two
consecutive edges $e_i$ and $e_{i+1}$ we have $|e_i\cap e_{i+1}|=1$
(where we consider $n+1=1$). It is easy to verify that if $n$ is not
divisible by $k-1$ then such a cycle cannot exist.

Frieze \cite{F} and Dudek and Frieze \cite{DF} showed that for
$p=\omega\left(\log n/n\right)$, the random $k$-uniform hypergraph
$H^{(k)}_{n,p}$ w.h.p.\ (with high probability) contains a loose
Hamilton cycle in $H^{(k)}_{n,p}$ whenever $2(k-1) | n$. Formally,
they showed:

\begin{theorem} \label{thm:frieze}The following hold:
\begin{enumerate}[$(a)$]
\item (Frieze) Suppose that $k=3$. Then there exists a constant $c>0$ such
that for $p\geq (c\log n)/n$ the following holds
$$\lim_{4|n\rightarrow \infty}\Pr\left[H^{(3)}_{n,p} \textrm{
contains a loose Hamilton cycle}\right]=1.$$

\item (Dudek and Frieze) Suppose that $k\geq 4$ and that
$pn^{k-1}/\log n$ tends to infinity. Then

$$\lim_{2(k-1)|n\rightarrow \infty}\Pr\left[H^{(k)}_{n,p} \textrm{ contains a loose Hamilton
cycle}\right]=1.$$
\end{enumerate}
\end{theorem}

The assumption $2(k-1) | n$ is clearly artificial, and indeed, in
\cite{DFP} Dudek, Frieze, Loh and Speiss removed it and showed
analog statement to \ref{thm:frieze} where there the only
restriction on $n$ is to be divisible by $k-1$ (which is optimal).

As a first result in this paper, we give a very short proof for the
result of Dudek, Frieze, Loh and Speiss while weakening $(a)$ a bit.
Formally, we prove the following theorem:

\begin{theorem} \label{main1}
The following hold:
\begin{enumerate}[$(a)$]
\item Suppose that $k=3$. Then for every $\varepsilon>0$ there exists a constant $c>0$ such
that for $p\geq (c\log n)/n$ the following holds
$$\lim_{2|n\rightarrow \infty}\Pr\left[H^{(3)}_{n,p} \textrm{
contains a loose Hamilton cycle}\right]\geq 1-\varepsilon.$$

\item Suppose that $k\geq 4$ and that
$pn^{k-1}/\log n$ tends to infinity. Then

$$\lim_{(k-1)|n\rightarrow \infty}\Pr\left[H^{(k)}_{n,p} \textrm{ contains a loose Hamilton
cycle}\right]=1.$$
\end{enumerate}
\end{theorem}

Another problem we handle with is the problem of finding a
\emph{rainbow} Hamilton cycle in a randomly edge-colored random
graph. For an integer $c$, let us denote by $G^{c}_{n,p}$ the random
graph $G_{n,p}$, where each of its edges is being colored, uniformly
at random with a color from $[c]$. A Hamilton cycle in $G^c_{n,p}$
is called \emph{rainbow} if all its edges receive distinct colors.
Clearly, a rainbow Hamilton cycle can not exists whenever $c<n$. Bal
and Frieze \cite{BF} showed that for some constant $K>0$, if $p\geq
(K\log n)/n$, the $G^n(n,p)$ w.h.p.\ contains a rainbow Hamilton
cycle, provided that $n$ is even. For the odd case, they proved
similar statement but for $p=\omega( (\log n)/n)$. We overcome this
and show the following:

\begin{theorem}\label{main2}
There exists a constant $K>0$ such that $G^n_{n,p}$ w.h.p.\ contains
a rainbow Hamilton cycle.
\end{theorem}

It is well known (see e.g. \cite{JLR}) that a Hamilton cycle appear
(w.h.p.) in $G_{n,p}$ for $p\approx (\log n)/n$. Therefore, one
would expect to prove an analog for Theorem \ref{main2} in this
range of $p$. However, it is easy to see that in this range, while
randomly color the edges of $G_{n,p}$ with $n$ colors, w.h.p. not
all the colors appear. Frieze and Loh \cite{FL} proved that for
$p=(1+\varepsilon)(\log n)/n$ and for $c=n+\Theta(n/\log\log n)$, a
graph $G^c_{n,p}$ w.h.p.\ contains a rainbow Hamilton cycle. It is
thus natural to consider the same problem for a randomly
edge-colored directed random graph, denoted by $D^c_{n,p}$ (we allow
edges to go in both directions). Note that in directed graphs we
require to have a directed Hamilton cycle, which is a Hamilton cycle
with all arcs pointing to the same direction.

The following theorem will follow quite immediately:

\begin{theorem}\label{main3}
Let $p=(1+\varepsilon)(\log n)/n$ and let $c=n+\Theta(n/\log\log
n)$. Then $D^c_{n,p}$ w.h.p.\ contains a rainbow Hamilton cycle.
\end{theorem}

Our proof is based on a very nice coupling argument due to McDiarmid
\cite{McD} and on Theorem \ref{thm:frieze}.

\section{Auxiliary results}

In this section we present some variants of a very nice argument by
McDiarmid \cite{McD}. For the convenient of the reader we add a
proof for one of them, and the rest will be left as easy exercises.
Before stating our lemmas, let us define the \emph{directed} random
$k$-uniform hypergraph $D^{(k)}_{n,p}$ in the following way. Each
ordered $k$-tuple $(x_1,\ldots,x_k)$ consisting of $k$ distinct
elements of $[n]$ appears as an \emph{arc} with probability $p$,
independently at random. In the special case where $k=2$ we simply
write $D_{n,p}$. A \emph{directed loose Hamilton cycle} is a loose
Hamilton cycle where consecutive vertices are now arcs of
$D^{(k)}_{n,p}$ and the last vertex of every are is the first of the
consecutive one. In the following lemma we show that the probability
for $D^{(k)}_{n,p}$ to have a directed loose Hamilton cycle is lower
bounded by the probability for $H^{(k)}_{n,p}$ to have one.

\begin{lemma} \label{lemma:colin}
Let $k\geq 3$. Then, for every $p:=p(n)\in (0,1)$ we have
$$\Pr\left[D^{(k)}_{n,p} \textrm{ contains a directed loose
Hamilton cycle}\right]\geq \Pr\left[H^{(k)}_{n,p} \textrm{ contains
a loose Hamilton cycle}\right].$$
\end{lemma}

\begin{proof} (McDiarmid)
Let us define the following sequence of random directed hypergraphs
$\Gamma_0,\Gamma_1,\ldots,\Gamma_N$, where $N=\binom{n}{k}$ in the
following way: Let $e_1,\ldots,e_N$ be an arbitrary enumeration of
all the (unordered) $k$-tuples contained in $[n]$. For each $e_i$
one can define $k!$ different orientations. Now, in $\Gamma_i$, for
every $j\leq i$ and for each of the $k!$ possible orderings of
$e_j$, we add the corresponding arc with probability $p$,
independently at random. For every $j>i$, we include all possible
orderings of $e_j$ or none with probability $p$, independently at
random. Note that $\Gamma_0$ is $H^{(k)}_{n,p}$ while $\Gamma_N$ is
$D^{(k)}_{n,p}$. Therefore, in order to complete the proof it is
enough to show that
\begin{align*}
&\Pr\left[\Gamma_i \textrm{ contains a directed loose Ham.\
cycle}\right]\geq \Pr\left[\Gamma_{i-1} \textrm{ contains a directed
loose Ham.\ cycle}\right].
\end{align*}

To this end, assume we exposed all arcs but those coming from $e_i$.
There are three possible scenarios:
\begin{enumerate}[$(a)$]
\item $\Gamma_{i-1}$ contains a directed loose Hamilton cycle without
considering $e_i$, or
\item $\Gamma_{i-1}$ does not contain a directed loose Hamilton
cycle even if we add all possible orderings of $e_i$, or
\item $\Gamma_{i-1}$ contains a directed loose Hamilton cycle using at
least one of the orderings of $e_i$.
\end{enumerate}

Note that in $(a)$ and $(b)$ there is nothing to prove. In case
$(c)$, the probability for $\Gamma_{i-1}$ to have a directed loose
Hamilton cycle is $p$, where the probability for $\Gamma_{i}$ to
have such a cycle is at least $p$. This completes the proof of the
lemma.
\end{proof}

In the second lemma, we show that given an integer $c$, one can
lower bound the probability of $D^c_{n,p}$ to have a rainbow
directed Hamilton cycle by the probability of $G^c_{n,p}$ to have
such a cycle.

\begin{lemma}\label{lem:rainDirected}
Let $c$ be a positive integer. Then, for every $p:=p(n)\in (0,1)$ we
have
$$\Pr\left[D^{c}_{n,p} \textrm{ contains a rainbow directed
Hamilton cycle}\right]\geq \Pr\left[G^{c}_{n,p} \textrm{ contains a
rainbow Hamilton cycle}\right].$$
\end{lemma}

Note that by combining the result of Bal and Frieze \cite{BF} with
Lemma \ref{lem:rainDirected} we immediately obtain the following
corollary:

\begin{corollary}\label{1}
There exists a constant $K>0$ such that for every $p\geq (K\log
n)/n$ we have
$$\Pr\left[D^n_{n,p} \textrm{ contains a rainbow Hamilton
cycle}\right]=1,$$ provided that $n$ is even.
\end{corollary}

\section{Proofs of our main results}

In this section we prove Theorems \ref{main1}, \ref{main2} and
\ref{main3}. We start with proving Theorem \ref{main1}.

\begin{proof}[Proof of Theorem \ref{main1}:] Suppose that $(k-1) | n$ and that $2(k-1)$ does not divide
$n$. Let $f^2(n)$ be a function that tends arbitrarily slowly to
infinity and suppose that $p=\frac{f^2(n)\log n}{n^{k-1}}$. Note
that by deleting the orderings of a $D^{(k)}_{n,q}$, using a similar
argument as a multi-round exposure (we refer the reader to
\cite{JLR} for more details), we obtain a $H^{(k)}_{n,s}$ where
$(1-q)^{k!}=1-s$ (one can just think about $D^{(k)}_{n,q}$ as an
undirected hypergraph such that for every $e\in \binom{n}{k}$ there
are $k!$ independent trials to decide whether to add it).

Now, let us choose $q$ in such a way that
$(1-p/2)(1-q)^{k!f(n)}=1-p$, and observe that $q\geq
\frac{p}{2k!f(n)}=\omega\left(\log n/n^{k-1}\right)$. We generate
$H^{(k)}_{n,p}$ in a multi-round exposure and present it as a union
$\bigcup_{i=0}^{f(n)}H_i$, where $H_0$ is $H^{(k)}_{n,p/2}$ and
$H_i$ is $D^{(k)}_{n,q}$ (which, as stated above, is like
$H^{(k)}_{n,s}$ with $(1-q)^{k!}=1-s$) for each $1\leq i\leq f(n)$
(of course, ignoring the orientations). In addition, all the $H_i$'s
are considered to be independent.

Our strategy goes as follows: First, take $H_0=H^{(k)}_{n,p/2}$ and
pick an arbitrary edge $e*=\{x_1\,\ldots,x_k\}$ (trivially, $H_0$
contains an edge w.h.p.). Now, fix an arbitrary ordering
$(x_1,\ldots,x_k)$ of $e^*$ and let $V^*=\left([n]\setminus
\{x_1,\ldots,x_k\}\right)\cup \{e^*\}$ (that is, $V^*$ is obtained
by deleting all the elements of $e^*$ and adding an auxiliary vertex
$e^*$). For each $i\geq 1$, whenever we expose $H_i$ we define an
auxiliary $k$-uniform directed random hypergraph $D_i$ on a vertex
set $V^*$ in the following way. Every arc $e$ of $H_i$ is being
added to $D_i$ if it satisfies one of the following:
\begin{itemize}
  \item $e\cap e^*=\emptyset$, or
  \item $e\cap e^*=\{x_1\}$, and $x_1$ is not the first vertex of the arc $e$, or
  \item $e\cap e^*=\{x_k\}$ and $x_k$ is the first vertex of the arc
  $e$.
\end{itemize}

Note that indeed, by definition, every $k$-tuple of $V^*$ now appear
with probability $p$, independently at random and that
$|V^*|=n-(k-1)$. Therefore, we clearly have that each of the $D_i$'s
is an independent $D^{(k)}_{n-(k-1),q}$. Moreover, note that
$2(k-1)|n$ and that each directed loose Hamilton cycle of $D_i$ with
the special vertex $e*$ as a starting/ending vertex of the edges
touching it corresponds to a (undirected) loose Hamilton cycle of
$H^{(k)}_{n,p}$. To see the latter, suppose that $e^*v_2\ldots v_t
e^*$ is such a cycle in $D_i$. Now, by definition we have that both
$x_kv_2,\ldots v_k$ and $v_{t-k+2}\ldots v_t e^*$ are arcs of $H_i$,
and therefore, by replacing $e^*$ with its entries $x_1\ldots x_k$,
one obtains a loose Hamilton cycle in $H_i$.

Next, by combining Theorem \ref{thm:frieze} with Lemma
\ref{lemma:colin}, we observe that w.h.p.\ $D_i$ contains a directed
loose Hamilton cycle. Note that by symmetry we have that the
probability for $e^*$ to be an endpoint of an edge on the Hamilton
cycle is $2/k$. Therefore, after exposing all the $D_i$'s, the
probability to fail in finding such a cycle is $(1-2/k)^{f(n)}=o(1)$
as desired. This completes the proof.
\end{proof}

Next we prove Theorem \ref{main2}.

\begin{proof}[Proof of Theorem \ref{main2}:]
Let us assume that $n$ is odd (since otherwise there is nothing to
prove) and that $K>$ is a sufficiently large constant for our needs.
Now, let $q$ be such that $(1-p/2)(1-q)^2=1-p$, and present
$G^n_{n,p}$ as a union $G_1\cup G_2$, where $G_1$ is $G^n_{n,p/2}$
and $G_2$ is $D^n_{n,q}$ (as in the proof of Theorem \ref{main1}, by
ignoring orientations one can see $D^n_{n,q}$ as $G^n_{n,s}$ with
$s$ satisfying $(1-q)^2=1-s$). Next, let $e^*=(x,y)$ be an arbitrary
edge of $G_1$ (trivially, w.h.p.\ there exists an edge), let $c_1$
denote its color, and define an auxiliary edge-colored random
directed graph $D$ as follows. The vertex set of $D$ is
$V^*=\left([n]\setminus {x,y}\right)\cup\{e^*\}$ (that is, we delete
$x$ and $y$ and add an auxiliary vertex $e^*$). The arc set of $D$
consist of all arcs $uv$ of $G_2$ with colors distinct than $c_1$
for which one of the following holds:

\begin{itemize}
\item $\{u,v\}\cap \{x,y\}=\emptyset$, or
\item $v=x$, or
\item $u=y$.
\end{itemize}

A moment's thought now reveals that $D$ is $D^{n-1}_{n-1,s}$, where
$s=(1-1/n)q$, that $n-1$ is even, and that a rainbow Hamilton cycle
of $D$ corresponds to a rainbow Hamilton cycle of $G^n_{n,p}$. Now,
since $s\geq (K'\log n)/n$ for some $K'$ (we can take it to be
arbitrary large), it follows from Corollary \ref{1} that w.h.p.\ $D$
contains a rainbow Hamilton cycle, and this completes the proof.
\end{proof}

Lastly, we prove Theorem \ref{main3}.

\begin{proof}[Proof of Theorem \ref{main3}:] The proof is an
immediate corollary of the result of Frieze and Loh \cite{FL} and
Lemma \ref{lem:rainDirected}.
\end{proof}

{\bf Acknowledgment.} The author would like to thank Alan Frieze for
helpful comments and for pointing out that there is also a small gap
in Bal and Frieze \cite{BF}.

\end{document}